\documentclass[12pt, a4paper]{amsart}

\usepackage[utf8]{inputenc}
\usepackage[T1]{fontenc}
\usepackage{amsmath,amssymb,amsthm}
\usepackage[margin=2.5cm]{geometry}
\usepackage{enumitem}
\usepackage{url} 
\usepackage{hyperref}

\usepackage{color}

\newtheorem{theorem}{Theorem}

\newtheorem{lemma}[theorem]{Lemma}

\theoremstyle{definition}

\theoremstyle{remark}

\theoremstyle{plain}
\newtheorem{thmA}{Theorem}

\theoremstyle{plain}
\newtheorem{conjA}[thmA]{Conjecture}

\newcommand{\PP}{\mathcal{P}}
\newcommand{\ord}{\mathrm{ord}}

\newcommand{\RR}{\mathcal{R}}
\newcommand{\PPP}{\mathbf{1}_{\mathcal{P}}}

\title{Positive density for Sun's $2^k+m$ conjecture
}

\author{Songlin Han}
\address{Graduate School of Mathematics, Kyushu University, Fukuoka, Japan.}
\email{han.songlin.638@s.kyushu-u.ac.jp}

\author{Jinbo Yu}
\address{Graduate School of Mathematics, Nagoya University, Nagoya, Japan.}
\email{jinbo.yu.e6@math.nagoya-u.ac.jp}


\date{\today}

\begin{document}

\begin{abstract}
  In 2013, Zhi-Wei Sun proposed a Romanov-type conjecture stating that every integer $n > 1$ can be written as $n = k + m$ with $k, m \ge 1$ such that $2^k + m$ is a prime. In this paper, we unconditionally prove that the natural numbers satisfying this property have a positive density. We compute this density to be at least $0.0734 $. We also discuss the limitations of our method. Under a uniform Hardy-Littlewood prime pairs conjecture, we show that the lower bound of density obtained by this method cannot exceed $1/(\log 2 + 1) \approx 0.5906$.
\end{abstract}

\maketitle

\section{Introduction}\label{sec:int}

Representation of integers as sums involving primes and powers of $2$ is a classical problem in additive number theory. In 1934, Romanov \cite{Romanov1934} proved that there is a positive density of natural numbers which can be expressed as the sum of a prime and a power of $2$. Let $\mathcal{P}$ be the set of all prime numbers.
\begin{thmA}[{\cite{Romanov1934}}, see also {\cite[Lemma 7.11]{Nathanson1996}}]
  Let $a \ge 2$ be an integer. Let
  \begin{equation*}
    A = \{y\geq 1\mid y=p+a^k \text{ for some } p\in \mathcal{P} \text{ and } k\ge 1\}
  \end{equation*}
  and $A(x) :=\# \{y \in A : y \le x\}$ be the counting function of $A$.
  There exists $c > 0$ such that for all sufficiently large $x$,
  \begin{equation*}
    A(x)>cx.
  \end{equation*}
\end{thmA}

In what follows, the letter $p$ always denotes a prime, that is, $p\in \mathcal{P}$.
Finding an explicit lower bound on this density has become an important problem. In 2004, Yong-Gao Chen and Xue-Gong Sun \cite{ChenSun2004} proved $c>0.0869$ by studying the related infinite series and gave a very important explicit estimate for its tail sum. Following their foundational work, in 2006, Pintz \cite{pintz2006note} made a significant improvement on this density, showing that $c>0.09367$. Meanwhile, Jie Wu \cite{Wu2004} improved the upper bounds for prime pairs by refining Chen's double sieve. Based on Wu's result, Elsholtz and Schlage-Puchta \cite{Elsholtz2014} later obtained a better lower bound $c\geq 0.1076$.

However, showing that such sets have positive density is distinct from proving that a representation holds for all sufficiently large integers. In 1950, Erd\H{o}s \cite{erdos1950integers} used covering systems to prove that a positive proportion of odd numbers cannot be written in the form $p+2^k$. This result was later generalized by Yong-Gao Chen. Similarly, Crocker \cite{crocker1971} showed that there are infinitely many odd integers that cannot be written as $p+2^a+2^b$, and Zhi-Wei Sun \cite{sun2000integers} further proved that a positive proportion of integers cannot be represented in the more general form $c(2^a+2^b)+p$. These examples show the difficulties in the study of sums of a prime and powers of an integer. The upper density of integers representable as $p+2^k$ has then been extensively studied; for instance, Habsieger and Roblot \cite{Habsieger2006} proved that the upper bound of the density is at most $0.49095$.

To avoid the issue caused by covering systems, Zhi-Wei Sun \cite{sun2014anbnmodulom} proposed the following modified Romanov-type conjecture in 2013.

\begin{conjA}[Sun's $2^k+m$ conjecture, \cite{sun2014anbnmodulom}]
\label{conj:sun}
For all integers $n>1$, there exist integers $k \ge 1$ and $m \ge 1$ such that $n = m + k$ and $2^k + m$ is a prime.
\end{conjA}

Sun numerically verified that this conjecture holds for all $n \le 10^7$ (see \cite{OEIS,sun2014anbnmodulom}). Although a complete proof of Conjecture~\ref{conj:sun} for all integers remains out of reach, its Romanov-type nature suggests that the set of integers satisfying this property should at least possess a positive lower density.

To investigate this density via the second-moment method, one inevitably encounters the distribution of prime pairs $p$ and $p+d$, where the difference $d$ belongs to a specific sparse sequence. A central object of this study is the average behavior of the singular series associated with the Hardy-Littlewood $k-$tuple conjecture \cite[p.14]{Broughan2021}. Recall that for an even integer $d$, the singular series is defined as
\begin{equation*}
    \mathfrak{S}(d) := 2C_2 \prod_{\substack{p|d\\p>2}}\left(1+\frac{1}{p-2}\right),
\end{equation*}
where $C_2 = \prod_{p\ge 3}(1-1/(p-1)^2)$ is the twin prime constant. 

In our context, we let $K= \lfloor\log_2 N - 2\log _2(\log N)\rfloor$, $d(k_2,h) = 2^{k_2}(2^h - 1 )- h$, and we define the corresponding mean value function $\Phi(N)$ over our specific index set as follows:
\begin{equation*}
    \Phi(N) := 
    \frac{1}{K^2}
    \sum_{\substack{1\le h \le K -1\\h \equiv 0 \bmod 2 }}\; 
    \sum_{1\le k_2\le K-h} \mathfrak{S}(d(k_2,h)).
\end{equation*}

The study of average values of the singular series is a classic topic in additive number theory. For instance, Gallagher \cite{Gallagher1976} proved that the average of $\mathfrak{S}(d)$ over consecutive integers in $[1, X]$ is asymptotically $1$. However, evaluating $\Phi(N)$ involves averaging over a double-indexed, exponentially sparse sequence, which presents unique arithmetic challenges related to the multiplicative order of $2$. We establish the following bounds:

\begin{theorem}\label{thm:phi_bounds}
The mean value function $\Phi(N)$ satisfies the following bounds:
\begin{equation*}
  \frac{1}{2} \le \liminf_{N\to\infty} \Phi(N) \le \limsup_{N\to\infty} \Phi(N) \le \frac{1}{2}\left( C_{\mathrm{Rom}} + e^{\gamma}\log 2 \right),
\end{equation*}
where $C_{\mathrm{Rom}}$ is the Romanov-type constant defined in \eqref{eq:CRom}, and $\gamma$ is the Euler constant.
\end{theorem}

By applying explicit estimates from the Selberg upper bound sieve and Theorem~\ref{thm:phi_bounds}, we refine the second-moment method to obtain a positive lower bound on the density as follows.

\begin{theorem}\label{thm:main}
Let $ \RR(N):= \bigl\{n \leq N : n = m + k \text{ and } 2^k + m \in \PP \text{ for some } k\ge 1,m\ge 1\bigr\}$.
For all sufficiently large $N$, we unconditionally have
\begin{equation*}
    \bigl|\RR(N)\bigr| \ge
    \left(
      \frac{1}{\log 2+8\Phi(N)} +o(1)
    \right)
    N.
\end{equation*}
In particular, there exists an absolute constant $\delta \ge 0.0734 $ such that $\bigl|\RR(N)\bigr| \geq \delta N$, meaning a positive proportion of natural numbers satisfy Sun's conjecture.
\end{theorem}

While Theorem~\ref{thm:main} unconditionally provides a positive lower bound, it is still far from the expected density $\delta=1$. To find an optimal bound achievable by this method, we assume a stronger form of the following Hardy-Littlewood prime pairs conjecture.
\begin{conjA}[Hardy-Littlewood prime pairs, \cite{Broughan2021}]
    Let $d\geq2$ be even. Then as $x\to\infty$, we have
    \begin{align}\label{eq:HLpp}
        \pi_d(x):=\sum_{\substack{p\leq x\\p-p'=d}}1=2C_2\prod_{\substack{p|d\\p>2}}\left(1+\frac{1}{p-2}\right)\frac{x}{(\log x)^2}\left(1+o\left(\frac{1}{(\log\log x)^2}\right)\right)
    \end{align}
    uniformly for $2\leq d\leq (\log x)^2$.
\end{conjA}

\begin{theorem}\label{thm:conditional}
  Assume that \eqref{eq:HLpp} for $\pi_d(x)$ holds uniformly for all even $d \le x$ as $x \to \infty$. Then, for all sufficiently large $N$, we have
  \begin{equation*}
    \bigl|\RR(N)\bigr| \ge
    \left(
      \frac{1}{\log 2+2\Phi(N)} +o(1)
    \right)
    N.
  \end{equation*}
  Consequently, combining this with the lower limit $\liminf_{N\to\infty}\Phi(N) \ge 1/2$, the bound on the lower density obtained by this method cannot exceed $\frac{1}{\log 2+1} \approx 0.5906$.
\end{theorem}

\section{Preparation}

By the relation $n = m+ k$, we first rewrite
\begin{equation*}
  2^k + m = 2^k - k + n .
\end{equation*}
If we define $c_k : = 2^k - k$, then clearly $c_{k+1}-c_k = 2^{k} - 1\ge 1$, so $\{c_k\}_{k\ge 1}$ is an increasing sequence.
Conjecture~\ref{conj:sun} says that there is $k$ with $1\le k \le n-1$ such that
\begin{equation*}
  2^k+(n-k) = n+c_k
\end{equation*}
is a prime.

Let $K=K(N):= \lfloor\log_2 N - 2\log _2(\log N)\rfloor$ for sufficiently large $N$.
This choice ensures $c_k\le 2^K\le N/(\log N)^2$ for all $1\le k\le K$.
Now, we define
\begin{equation*}
  r(n):=\sum_{1\le k\le K}\PPP(n + c_k)
\end{equation*}
for integers $n$ with $K+1\le n\le N$, where
\begin{equation*}
  \PPP(n):=
  \begin{cases}
    1&\text{if }n\in\PP,\\
    0&\text{if }n\notin\PP,
  \end{cases}
\end{equation*}
and so
\begin{equation*}
  \RR^{\ast}(N)= \{K+1\le n\le N: r(n)\ge 1\},
\end{equation*}
where $\RR^{\ast}(N)$ is a subset of $\RR(N)$ defined in Section 1, which counts the numbers satisfying Conjecture~\ref{conj:sun} in the range $[K+1,N]$.
Define
\begin{equation*}
  S_1(N):=\sum_{K+1\le n \le N}r(n),\quad S_2(N):=\sum_{K+1\le n \le N}r^2(n),
\end{equation*}
then by Cauchy-Schwarz inequality, we have
\begin{equation}\label{Cauchy-Schwarz}
  \bigl|\RR(N)\bigr|\ge\bigl|\RR^{\ast}(N)\bigr|\ge \frac{S_1(N)^2}{S_2(N)}.
\end{equation}
Our idea is to prove Theorem~\ref{thm:main} via the following lemma.
\begin{lemma}\label{prop:main}
  There exist absolute constants $c_1>0$ and $c_2>0$ such that $S_1(N)\ge c_1 N$, and $S_2(N)\le c_2 N$ for sufficiently large $N$.
\end{lemma}
Thus by \eqref{Cauchy-Schwarz}, Theorem~\ref{thm:main} is proved if we set $\delta := c_1^2/c_2>0$.

\section{Proof of Lemma~\ref{prop:main}}
\label{sec:pflem}

In this section, we prove the proposition by finding a lower bound for $S_1$ and an upper bound for $S_2$.
For $S_1$, we may exchange the order of summation to derive
\begin{equation*}
  S_1(N) = \sum_{1\le k\le K}\; \sum_{K+1\le n\le N} \PPP (n+c_k).
\end{equation*}
For each fixed $k$, as $n$ runs over $K+1$ to $N$, $c_k+n$ runs over $K+1+c_k$ to $N+c_k$,
thus 
\begin{equation*}
  \sum_{K+1\le n\le N}\PPP(n+c_k) = \pi(N+c_k)-\pi(K+c_k)
\end{equation*}
where $\pi(x) := \sum_{ p\le x}1$ is the prime counting function.

By the prime number theorem and
\begin{align*}
    \pi(N+c_k)-\pi(N)\leq c_k\leq\frac{N}{(\log N)^2},
\end{align*}
we have
\begin{align*}
    \pi(N+c_k)=\frac{N}{\log N}+O\left(\frac{N}{(\log N)^2}\right).
\end{align*}
Since $K+c_k\leq K+c_K=2^K \leq\frac{N}{(\log N)^2}$,
\begin{align*}
    \pi(K+c_k)\leq K+c_k\leq \frac{N}{(\log N)^2},
\end{align*}
which gives the inner sum of $S_1(N)$ as
\begin{align*}
    \sum_{K+1\le n\le N}\PPP(n+c_k)=\pi(N+c_k)-\pi(K+c_k)=\frac{N}{\log N}+O\left(\frac{N}{(\log N)^2}\right).
\end{align*}
Therefore, we have
\begin{align*}
    S_1(N)=\sum_{k=1}^K\left(\frac{N}{\log N}+O\left(\frac{N}{(\log N)^2}\right)\right)
   =K\frac{N}{\log N}+O\left(K\frac{N}{(\log N)^2}\right)
\end{align*}
Since $K=\frac{\log N}{\log 2}+O(\log\log N)$, then
\begin{align*}
    S_1(N)=\frac{N}{\log2}+O\left(N\frac{\log\log N}{\log N}\right).
\end{align*}

Now we turn to the upper bound for $S_2$.
We expand $r(n)^2$ and separate the contributions into two parts
\begin{equation}
    \begin{split}
     S_2(N) =\sum_{1\le k\le K} \; \sum_{K+1\le n \le N}\PPP (n+c_k) + \underset{\substack{1\le k_1\le K\\1\le k_2\le K\\k_1\neq k_2}}{\sum\sum}
    \; \sum_{K+1\le n\le N}\PPP(n+c_{k_1})\PPP(n+c_{k_2}).
    \label{eq:S2}
    \end{split}
\end{equation}
The first sum in \eqref{eq:S2} is exactly $S_1(N)$, and we denote the second sum in \eqref{eq:S2} as $D(N)$.
To complete the proof, it suffices to show that $D(N) = O(N)$.

By symmetry, we can write 
\begin{equation*}
    D(N) = 2\sum_{1\le h\le K-1}\; \sum_{1\le k_2 \le K-h} F(k_2, h)
\end{equation*}
where
\begin{equation*}
    F(k_2,h):= \sum_{K+1\le n\le N}\PPP(n+c_{k_2+h})\PPP(n+c_{k_2}).
\end{equation*}
Define the difference
\begin{equation*}
    d = d(k_2,h):=c_{k_2+h}-c_{k_2}=2^{k_2}(2^h-1)-h,
\end{equation*}
which depends only on $k_2$ and $h$.
Since $k_2\ge 1$, the term $2^{k_2}(2^h-1)$ is even, so $d\equiv -h \bmod 2$.
One can verify that $1\le d\le 2^K\le N/(\log N)^2$.

If $h$ is odd, then $d$ is odd as well, and $n+c_{k_2}$, $n+c_{k_2}+d$ have opposite parity.
Since $n\ge K+1$ and $c_{k_2}\ge c_1 = 1$, we have $n+c_{k_2}\ge K+2\ge 3$ and $n+c_{k_2}+d\ge K+3\ge 4$ for sufficiently large $N$.
Hence neither can equal to $2$, and the even one of the pair exceeds $2$, so it cannot be prime.
This gives $F(k_2,h) = 0$ for all odd $h$.

Let $\mathcal{I}_K:=\{(k_2,h): 1\le h\le K-1,\ h\text{ even},\ 1\le k_2\le K-h\}$,
so that $|\mathcal{I}_{K}| = \frac{K^2}{4}+O(K)$.
The above shows that 
\begin{equation}\label{eq:D1_bound}
    D(N) = 2\sum_{(k_2,h)\in\mathcal{I}_{K}}F(k_2, h ).
\end{equation}

Now we turn to the case $d$ is even.
Let $\pi_d(x)$ be the counting function defined in \eqref{eq:HLpp}.
Since $F(k_2,h)$ counts primes $p\le N+c_{k_2}$ such that $p+d$ is also prime (where $p=n+c_{k_2}$, $d$ depends on $k_2$ and $h$), we have $F(k_2,h)\le \pi_d(N+c_{k_2})$ with $d=d(k_2,h)$.
In this case, the upper bound of Selberg sieve (see Halberstam and Richert \cite[Theorem 3.11]{HalberstamRichert}) asserts that for $x\ge 2$,
\begin{equation*}
  \pi_d(x)\leq 8\prod_{p>2}\left(1-\frac{1}{(p-1)^2}\right) \prod_{\substack{p|d \\ p>2}}\frac{p-1}{p-2}\frac{x}{\log^2x}\left(1+O\left(\frac{\log\log x}{\log x}\right)\right),
\end{equation*}
uniformly in $d$.
Here, $\prod_{p>2}\left(1-\frac{1}{(p-1)^2}\right)=C_2$ is the twin prime constant defined in Section \ref{sec:int}. 
By the definition of the singular series $\mathfrak{S}(d)$, we can rewrite the explicit upper bound for sufficiently large $N$ as
\begin{equation}\label{eq:F_bound_explicit}
  F(k_2,h) \le 4\mathfrak{S}(d(k_2,h)) \frac{N}{(\log N)^2} \left(1+O\left(\frac{\log\log N}{\log N}\right)\right).
\end{equation}
Substituting \eqref{eq:F_bound_explicit} back into \eqref{eq:D1_bound} yields the upper bound
\begin{equation*}
  D(N) \le \frac{8N}{(\log N)^2} \left(\sum_{(k_2,h)\in\mathcal{I}_K} \mathfrak{S}(d(k_2,h))\right) (1+o(1)).
\end{equation*}
By the definition of $\Phi(N)$, the inner summation is exactly $K^2 \Phi(N)$.
Recalling that our truncation gives $K \sim \log N / \log 2$, we have $K^2 \sim (\log N)^2 / (\log 2)^2$. Thus, we obtain
\begin{equation*}
  D(N) \le \frac{8\Phi(N)}{(\log 2)^2} N (1+o(1)).
\end{equation*}
Since the diagonal contribution is exactly $S_1(N) \sim \frac{N}{\log 2}$, we conclude that 
\begin{equation*}
  S_2(N) = S_1(N) + D(N) \le \left(\frac{1}{\log 2} + \frac{8\Phi(N)}{(\log 2)^2}\right) N (1+o(1)).
\end{equation*}
By Theorem~\ref{thm:phi_bounds}, $\Phi(N)$ is bounded as $N \to \infty$. Thus, there exists an absolute constant $c_2 > 0$ such that $S_2(N) \le c_2 N$ for sufficiently large $N$. This completes the proof of Lemma~\ref{prop:main}. 
\qed

\section{Proof of Theorem~\ref{thm:phi_bounds}}
\label{sec:phi_bounds}

In this section, we establish the absolute upper and lower bounds for the mean value function $\Phi(N)$ over our sequence.

\subsection{Upper bound of \texorpdfstring{$\Phi(N)$}{Phi(N)}}

Let $g$ be the multiplicative function supported on odd square-free integers defined by $g(p) = \frac{1}{p-2}$. We can write the inner product in singular series $\mathfrak{S}(d)$ as
\begin{equation*}
   f(d):= \prod_{\substack{p|d\\p>2}}\frac{p-1}{p-2}=\prod_{\substack{p|d\\p>2}}\left(1+\frac{1}{p-2}\right)=\sum_{\substack{l|d \\ 2\nmid l}} \mu^2(l)g(l).
\end{equation*}
Denote 
\begin{equation*}
    \sigma := \sum_{(k_2,h)\in\mathcal{I}_K} \sum_{\substack{l|d \\ 2\nmid l}} \mu^2(l)g(l),
\end{equation*}
we can write $\Phi(N)$ as
\begin{equation*}
  \Phi(N) = \frac{2C_2}{K^2} \sigma.
\end{equation*}

To evaluate $\sigma$, we proceed by first fixing an even $h$ and summing over $k_2$, defining
\begin{equation*}
  \sigma_h := \sum_{1\le k_2 \le K-h} \sum_{\substack{l|d \\ 2\nmid l}} \mu^2(l)g(l) = \sum_{\substack{l \ge 1 \\ 2\nmid l}} \mu^2(l)g(l) M_l(h),
\end{equation*}
where $M_l(h) := \#\{1\le k_2 \le K-h\mid 2^{k_2}(2^h-1)\equiv h \bmod l\}$.

For each fixed even $h$ and odd square-free $l\ge 1$, we write  $b=l/\gcd(l,2^h-1)$.
By analyzing the congruence conditions modulo $\gcd(l,2^h-1)$ and modulo $b$ separately, and applying the Chinese Remainder Theorem, one finds that 
\begin{equation*}
  M_l (h) 
  \begin{cases}
    =0&\text{if }\gcd(l,2^h-1)\nmid h,\\
    \le \frac{K-h}{\ord_b(2)}+1&\text{if }\gcd(l,2^h-1)\mid h,
  \end{cases}
\end{equation*}
where $\ord_b(2)$ is the multiplicative order of $2$ modulo $b$.

Parameterizing the sum by $l = \gamma\beta$ with $\gamma\mid\gcd(2^h-1,h)$ and $\gcd(\gamma,\beta)=1$. By the multiplicativity of $g$, we have $\mu^2(l)g(l) = \mu^2(\gamma)g(\gamma)\mu^2(\beta)g(\beta)$. Thus, we obtain
\begin{align*}
  \sigma_h \le &\sum_{\substack{l\geq1,\,2\nmid l\\\gcd(l,2^h-1)\mid h}}\mu^2{(l)}g(l)\left(\frac{K-h}{\ord_b(2)}+1\right)\\
  &=\sum_{\substack{l\geq1,\,2\nmid l\\\gcd(l,2^h-1)\mid h}}\mu^2{(l)}g(l)\left(\frac{K-h}{\ord_b(2)}\right)+\sum_{\substack{l\geq1,\,2\nmid l\\\gcd(l,2^h-1)\mid h}}\mu^2{(l)}g(l)\\
  &\leq\underbrace{\left(\sum_{\substack{\gamma \mid h \\ 2\nmid \gamma}} \mu^2(\gamma)g(\gamma)\right)}_{=\,f(h)} \cdot \left( \sum_{\substack{\beta \le 2^K \\ 2\nmid \beta}} \mu^2(\beta)g(\beta) \frac{K-h}{\ord_{\beta}(2)} \right)+\sum_{\beta\leq 2^K,\,2\nmid \beta}\mu^2(\beta)g(\beta).   
\end{align*}

We split the inner sum into a main term governed by the Romanov-type constant 
\begin{equation}\label{eq:CRom}
    C_{\mathrm{Rom}} := \sum_{2\nmid \beta} \frac{\mu^2(\beta)g(\beta)}{\ord_{\beta}(2)},
\end{equation}
and an error term coming from the $+1$. 
Note that $\beta \mid d$ implies $\beta \le d$. Since $k_2 + h \le K$, we have $d < 2^{k_2+h} \le 2^K$, which ensures $\beta < 2^K$. 
Here $C_{\mathrm{Rom}}$ is a convergent absolute constant, which we refer to as the Romanov-type constant (see Nathanson \cite[Lemma 7.8]{Nathanson1996}). It comes from the part involving the divisors of $2^n-1$. The absolute convergence of this series was first established in Romanov's seminal work on the representation of integers as the sum of a prime and a power of $2$. Thus, we obtain
$$\sigma_h \le f(h) \Big( (K-h) C_{\mathrm{Rom}} \Big)+ \sum_{\substack{\beta \le 2^K \\ 2\nmid \beta}} \mu^2(\beta)g(\beta) .$$

Recall that in the double sum over the index set $\mathcal{I}_K$, for each fixed even $h = 2j$, the inner index $k_2$ ranges from $1$ to $K-h$. Therefore, the inner summation over $k_2$ naturally produces a multiplier of $(K-2j)$, which represents the exact length of the available interval. This constraint restricts our summation to a triangular domain $\left\{ (j, k_2) \in \mathbb{Z}^2 : 1 \le j \le \frac{K-1}{2}, \ 1 \le k_2 \le K-2j \right\}.$
To evaluate the total sum over this triangular domain, we first consider the simple linear average of $f(2j)$ up to a real boundary $X \ge 1$. Reversing the order of summation, we extract the odd square-free divisors $\eta$ of $j$,
\begin{equation*}
  \sum_{j \le X} f(2j) = \sum_{j \le X} \sum_{\substack{\eta|j \\ 2\nmid \eta}} \mu^2(\eta)g(\eta) \le X \sum_{\substack{\eta \ge 1 \\ 2\nmid \eta}} \frac{\mu^2(\eta)g(\eta)}{\eta}.
\end{equation*}
Surprisingly, this infinite sum is precisely the inverse of the twin prime constant
\begin{equation*}
  \sum_{\substack{\eta \ge 1 \\ 2\nmid \eta}} \frac{\mu^2(\eta)g(\eta)}{\eta} = \prod_{p>2} \left(1+\frac{1}{p(p-2)}\right) = \prod_{p>2} \frac{(p-1)^2}{p(p-2)} = \frac{1}{C_2}.
\end{equation*}
Thus, the linear sum is bounded by  
\begin{equation}
\label{eq:even_h_bound}
    \sum_{j \le X} f(2j) \le \frac{1}{C_2}X. 
\end{equation}
Then the sum over the triangular domain can be evaluated by applying partial summation,
\begin{equation*}
\begin{split}
  \sum_{j=1}^{\lfloor (K-1)/2 \rfloor} f(2j)(K-2j) &= \int_{1^{-}}^{K/2} (K-2t) \, d\left( \sum_{j \le t} f(2j) \right) \\
  &= \left[ (K-2t) \sum_{j \le t} f(2j) \right]_{1^{-}}^{K/2} + \int_{1}^{K/2} 2\left( \sum_{j \le t} f(2j) \right)  \, dt \\
  &\le 0 + \int_{0}^{K/2} 2\left( \frac{1}{C_2} t \right)  \, dt = \frac{1}{C_2} \left[ t^2 \right]_0^{K/2} = \frac{1}{4 C_2} K^2.
\end{split}
\end{equation*}

To bound the error term arising from the $+1$ in the inner summation, we evaluate the sum of the multiplicative function $g(\beta)$ over odd square-free integers $\beta \le 2^K$,
\begin{equation*}
  \sum_{\substack{\beta \le 2^K \\ 2\nmid \beta}} \mu^2(\beta)g(\beta) \le \prod_{2 < p \le 2^K} \left( 1 + \frac{1}{p-2} \right) = \prod_{2 < p \le 2^K} \frac{p-1}{p-2}.
\end{equation*}
We decompose each factor to explicitly extract the twin prime constant $C_2$ and the classical Mertens product:
\begin{equation*}
  \frac{p-1}{p-2} = \left( 1 - \frac{1}{p} \right)^{-1} \left( 1 - \frac{1}{(p-1)^2} \right)^{-1}.
\end{equation*}
Taking the product over all odd primes $p \le 2^K$, the second factor converges rapidly to the inverse of the twin prime constant:
\begin{equation*}
  \prod_{2 < p \le 2^K} \left( 1 - \frac{1}{(p-1)^2} \right)^{-1} < \prod_{p > 2} \left( 1 - \frac{1}{(p-1)^2} \right)^{-1} = \frac{1}{C_2}.
\end{equation*}
For the first factor, Mertens' third Theorem gives that
\begin{align*}
    \prod_{2 < p \le X} \left(1 - \frac{1}{p}\right)^{-1} \sim \frac{e^{\gamma}}{2} \log X,
\end{align*}
where $\gamma$ is the Euler constant. 
We obtain
\begin{equation*}
    \prod_{2 < p \le 2^K} \left( 1 - \frac{1}{p} \right)^{-1} \le \frac{e^{\gamma}}{2}\log\left(2^K\right)(1+o(1)).
\end{equation*}
Therefore, we get the following upper bound for $\sigma$:
\begin{equation*}
\begin{split}
  \sigma \le \left(\frac{1}{4 C_2} K^2 C_{\mathrm{Rom}} + \frac{e^{\gamma} \log 2}{4 C_2} K^2 \right)(1+o(1)) = \frac{K^2}{4 C_2} \left( C_{\mathrm{Rom}} + e^{\gamma}\log 2 \right) (1+o(1)).
\end{split}
\end{equation*}
Substituting this explicit bound for $\sigma$ back into our formula $\Phi(N) = \frac{2C_2}{K^2} \sigma$, we observe a remarkable exact cancellation of $K^2$ and the twin prime constant $C_2$:
\begin{equation*}
  \limsup_{N\to\infty} \Phi(N) \le \frac{1}{2} \left( C_{\mathrm{Rom}} + e^{\gamma}\log 2\right).
\end{equation*}

\subsection{Lower bound of \texorpdfstring{$\Phi(N)$}{Phi(N)}}
To prove the lower bound of $\Phi(N)$, we use the following lemma regarding the equidistribution of $d(k_2,h)$ modulo $p$.
\begin{lemma}\label{lem:equidist}
  For each odd prime $p$,
  \begin{equation*}
    \rho_p(K):=\frac{|\{(k_2,h)\in\mathcal{I}_K: p\mid d(k_2,h)\}|}{|\mathcal{I}_K|}
    = \frac{1}{p}+O_p(K^{-1}).
  \end{equation*}
\end{lemma}

\begin{proof}[Proof of Lemma~\ref{lem:equidist}]
  Since $h$ is restricted to be even in $\mathcal{I}_K$, we parameterize $h = 2j$.
  The condition $p \mid d(k_2, h)$ becomes
  \begin{equation}\label{eq:congcond}
    2^{k_2}(4^j - 1) \equiv 2j \pmod{p}.
  \end{equation}
  The tuple $(2^{k_2} \bmod p,\, 4^j \bmod p,\, 2j \bmod p)$ is periodic in $(k_2, j)$ with period $(\ord_p(2),\, p\cdot\ord_p(4))$, since $\gcd(\ord_p(4), p) = 1$.
  We count solutions of \eqref{eq:congcond} in one fundamental domain $k_2 \in \{0, \dots, \ord_p(2)-1\}$, $j \in \{0, \dots, p\cdot\ord_p(4) - 1\}$.

  \smallskip
  We first consider the case $\ord_p(4) \mid j$ (so $4^j \equiv 1 \pmod{p}$).
  Then \eqref{eq:congcond} reduces to $p \mid 2j$, hence $p \mid j$.
  Since $\gcd(\ord_p(4), p) = 1$, the joint condition $\ord_p(4) \mid j$ and $p \mid j$ forces $p\cdot\ord_p(4)\mid j$,
  giving exactly $1$ value of $j$ (namely $j=0$) and all $\ord_p(2)$ values of $k_2$.
  So the contribution of this case is $\ord_p(2)$.

  Now we deal with the case $\ord_p(4) \nmid j$ and $p \mid j$.
  Then $4^j \not\equiv 1$ and \eqref{eq:congcond} becomes $2^{k_2}(4^j - 1) \equiv 0 \pmod{p}$, impossible since $p\nmid 2^{k_2}$ and $p\nmid(4^j-1)$.
  So the contribution of this case is $0$.

  Finally, we turn to the case $\ord_p(4) \nmid j$ and $p \nmid j$.
  Then \eqref{eq:congcond} becomes $2^{k_2} \equiv 2j(4^j - 1)^{-1} \pmod{p}$.
  There are $\ord_p(4) - 1$ values of $j \bmod \ord_p(4)$ with $4^j \not\equiv 1$.
  For each, as $j$ varies mod~$p$, the target $2j(4^j-1)^{-1}\bmod p$ ranges over all $p-1$ nonzero residues
  (since $\gcd(\ord_p(4),p)=1$ and multiplication by $2(4^j-1)^{-1}\not\equiv 0$ is a bijection).
  Exactly $\ord_p(2)$ of the $p-1$ residues lie in the group generated by $2$, each determining $k_2$ uniquely mod~$\ord_p(2)$.
  So in this case, the contribution is $(\ord_p(4)-1)\cdot\ord_p(2)$.

  In total, $\ord_p(2)+(\ord_p(4)-1)\cdot\ord_p(2) = \ord_p(4)\cdot\ord_p(2)$ solutions out of $\ord_p(2)\cdot p\cdot\ord_p(4) = p\cdot\ord_p(4)\cdot\ord_p(2)$ pairs, giving density $1/p$.
  Boundary effects contribute $O_p(K)$ to $|\mathcal{I}_K|\rho_p(K)$, hence $\rho_p(K)=1/p+O_p(K^{-1})$.
\end{proof}

For convenience, we define $f_P(d)$ as follows. For prime $P>2$, define
\begin{equation*}
  f_P(d) := \prod_{\substack{p|d\\2<p\le P}}\frac{p-1}{p-2},
\end{equation*}
so that $f(d)\ge f_P(d)$.
Since $f_P$ depends only on $d\bmod Q_P$ where $Q_P:=\prod_{2<p\le P}p$, we can apply the Chinese remainder theorem. This theorem ensures that the congruence conditions modulo distinct primes $p \mid Q_P$ are independent. Therefore,we can naturally extend Lemma~\ref{lem:equidist} to the square-free modulus $Q_P$ to compute the average of $f_P$ over $\mathcal{I}_K$ for any fixed $P$:
\begin{align*}
  \frac{1}{|\mathcal{I}_K|}\sum_{(k_2,h)\in\mathcal{I}_K}f_P(d)
  &= \prod_{2<p\le P}\left(\frac{1}{p}\cdot\frac{p-1}{p-2}+\frac{p-1}{p}\right)+O_P(K^{-1})
  \\ &= \prod_{2<p\le P}\left(1+\frac{1}{p(p-2)}\right)+O_P(K^{-1})
\end{align*}
as $K$ tends to infinity.
Therefore
\begin{equation*}
  \liminf_{K\to\infty}\frac{1}{|\mathcal{I}_K|}\sum_{(k_2,h)\in\mathcal{I}_K} f_P(d)\ge \prod_{2<p\le P}\left(1+\frac{1}{p(p-2)}\right).
\end{equation*}
Since the left-hand side is independent of $P$ as we let $P\to\infty$, we obtain
\begin{equation*}
  \liminf_{K\to\infty}\frac{1}{|\mathcal{I}_K|}\sum_{(k_2,h)\in\mathcal{I}_K} f(d)
  \ge \prod_{p>2}\left(1+\frac{1}{p(p-2)}\right)
  = \prod_{p>2}\frac{(p-1)^2}{p(p-2)}
  = \frac{1}{C_2}.
\end{equation*}
Since $\Phi(N) = \frac{2C_2}{K^2}\sum_{(k_2,h)\in\mathcal{I}_K} f(d)$ and $|\mathcal{I}_K| \sim \frac{K^2}{4}$, this gives $\liminf_{N\to\infty}\Phi(N)\ge \frac{1}{2}$.
\qed

\section{Lower bound of the density}
\label{sec:DensityProofs}

\subsection{Unconditional density \texorpdfstring{$\delta$}{delta}}
Now we can explicitly determine the lower bound for the density $\delta$ in Theorem \ref{thm:main}. Recall that $\RR(N)$ denotes the set of integers $n \in [1, N]$ satisfying Conjecture~\ref{conj:sun}. Since $\RR^{\ast}(N) \subseteq \RR(N)$ and $|\RR(N) \setminus \RR^{\ast}(N)| \le K = O(\log N)$, we have
\begin{equation*}
  |\RR(N)| \ge |\RR^{\ast}(N)| \ge \frac{S_1(N)^2}{S_2(N)}.
\end{equation*}
From our previous discussion, the first moment is
\begin{equation*}
  S_1(N) \sim \frac{1}{\log 2} N.
\end{equation*}
For the second moment, we have established the strictly explicit upper bound $S_2(N) \le C_{\text{upper}} N$, where the constant $C_{\text{upper}}$ is given by
\begin{equation*}
 C_{\text{upper}} = \frac{1}{\log 2} + \frac{8}{(\log 2)^2} \limsup_{N\to\infty} \Phi(N) \le \frac{1}{\log 2} + \frac{4}{(\log 2)^2} \left( C_{\mathrm{Rom}} + e^{\gamma} \log 2 \right).
\end{equation*}

To establish an explicit bound for $C_{\mathrm{Rom}}$, we introduce a truncation parameter $M \ge 1$ and decompose the series according to the multiplicative order $k = \ord_\beta(2)$ as follows:
\begin{equation*}
  C_{\mathrm{Rom}} = \sum_{\substack{2\nmid \beta \\ \ord_\beta(2) \le M}} \frac{\mu^2(\beta)g(\beta)}{\ord_\beta(2)} + \sum_{\substack{2\nmid \beta \\ \ord_\beta(2) > M}} \frac{\mu^2(\beta)g(\beta)}{\ord_\beta(2)}.
\end{equation*}
The first sum is finite and can be evaluated directly. We follow the numerical evaluation in \cite{Elsholtz2014}. By grouping the terms according to $k = \ord_\beta(2)$ and applying M\"obius inversion over the odd square-free divisors of $2^k-1$, an exact computation for $M = 240$ yields that this main term is strictly bounded by $1.9328$.

For the tail of the series, we appeal to an estimate established by Chen and Sun \cite[Lemma 4]{ChenSun2004} (see also \cite{Elsholtz2014}). They proved that for any $M \ge 1$,
\begin{equation*}
  \sum_{\substack{2\nmid \beta \\ \ord_\beta(2) > M}} \frac{g(\beta)}{\ord_\beta(2)} \le 2.7961 \frac{\log M}{M}.
\end{equation*}
Taking $M = 240$, the contribution of the tail is explicitly bounded by $2.7961 \frac{\log 240}{240} < 0.0639$. Combining these estimates, we unconditionally conclude that
\begin{equation*}
  C_{\mathrm{Rom}} < 1.9328 + 0.0639 < 1.9967.
\end{equation*}

Now we explicitly determine the lower bound for the density $\delta$ stated in Theorem \ref{thm:main}. Using the classical values $\log 2 \approx 0.69315$, Euler's constant $\gamma \approx 0.57722$ (which gives $e^\gamma \approx 1.78107$), and our bound $C_{\mathrm{Rom}} < 1.9967$, the term in the parenthesis evaluates to:
\begin{equation*}
  C_{\mathrm{Rom}} + e^{\gamma}\log 2 < 1.9967 + 1.78107 \times 0.69315 \approx 1.9967 + 1.23456 = 3.23126.
\end{equation*}
Thus, the upper bound constant of $S_2(N)$ is strictly controlled by
\begin{equation*}
  C_{\text{upper}} < 1.4427 + \frac{4}{(0.69315)^2} \times 3.23126 \approx 1.4427 + 8.3255 \times 3.23126 \approx 1.4427 + 26.9019 \le = 28.3446.
\end{equation*}
Substituting these explicit values back into the Cauchy-Schwarz inequality, we obtain the unconditional lower bound for the density:
\begin{equation*}
 \liminf_{N \to \infty} \frac{|\mathcal{R}(N)|}{N} \ge \frac{(1/\log 2)^2}{28.3446} \approx \frac{2.0814}{28.3446} > 0.0734 
\end{equation*}
Therefore, we have unconditionally proven that the lower density of natural numbers satisfying Conjecture~\ref{conj:sun} is at least $7.34\%$.

\subsection{Density under the uniformly prime pairs conjecture}
In this section, we discuss the limitations of our method under the uniform version of the Hardy-Littlewood prime pairs conjecture as stated in Theorem~\ref{thm:conditional}.
Since $d(k_2,h)\le N/(\log N)^2$ for all $(k_2,h)\in\mathcal{I}_K$, the uniform assumption ensures that
\begin{equation*}
    \pi_d(N+c_{k_2}) = \mathfrak{S}(d)\frac{N}{(\log N)^2}(1+o(1))
\end{equation*}
holds uniformly for all pairs $(k_2,h)\in\mathcal{I}_K$ as $N$ tends to infinity.
Moreover, $\pi_d (K+c_{k_2})\le \pi(K+c_{k_2}) = O(N/(\log N)^2)$ by the prime number theorem, so
\begin{equation*}
    F(k_2,h) = \pi_d(N+c_{k_2})-\pi_d(K+c_{k_2})
    =\mathfrak{S}(d)\frac{N}{(\log N)^2}(1+o(1))
\end{equation*}
uniformly for all $(k_2,h)\in\mathcal{I}_K$.

For $D(N)$, by our assumption and \eqref{eq:D1_bound}, we have
\begin{align*}
        D(N) 
        &= \frac{2N}{(\log N)^2}\left(\sum_{(k_2,h)\in\mathcal{I}_K}\mathfrak{S}(d(k_2,h))\right)(1+o(1)).
\end{align*}
Then by the definition of $\Phi(N)$, we have $\sum_{(k_2,h)\in\mathcal{I}_K}\mathfrak{S}(d(k_2,h)) = K^2\Phi(N)$.
Since $K=\frac{\log N}{\log 2}(1+o(1))$, this gives us 
\begin{equation*}
    D(N) = \frac{2\Phi(N)N}{(\log 2)^2}(1+o(1)).
\end{equation*}
Since $S_2(N) = S_1(N)+D(N)$, \eqref{Cauchy-Schwarz} gives

\begin{equation*}
    \bigl| \RR(N) \bigr| \ge \frac{S_1(N)^2}{S_1(N)+D(N)}\ge \left( \frac{1}{\log 2+2\Phi(N)} +o(1) \right) N.
\end{equation*}
So a smaller value of $\Phi(N)$ yields a better density bound.
Using the fact that
$$\liminf_{N\rightarrow\infty}\Phi(N)\ge \frac{1}{2}$$
as proven in Section~\ref{sec:phi_bounds}, this shows a limitation of our method: even if we assume the uniform Hardy-Littlewood conjecture, our method can only give a density bound of at most
$$\frac{1}{\log 2+2(1/2)} = \frac{1}{\log 2+1} \approx 0.5906\dots.$$

\section*{Acknowledgement}
The authors would like to thank Henrik Bachmann, Kohji Matsumoto, Ade Irma Suriajaya, and Yuta Suzuki for their kind advice. We would also like to thank Zhi-Wei Sun for introducing this conjecture to us. J.~Yu is financially supported by JST SPRING, Grant Number JPMJSP2125, and would like to take this opportunity to thank the "THERS Make New Standards Program for the Next Generation Researchers".

\nocite{*}
\bibliographystyle{amsplain}
\bibliography{ref}

@book{Nathanson1996,
  author    = {Nathanson, Melvyn B.},
  title     = {Additive Number Theory: The Classical Bases},
  series    = {Graduate Texts in Mathematics},
  volume    = {164},
  publisher = {Springer-Verlag},
  address   = {New York},
  year      = {1996},
  isbn      = {0-387-94656-X},
  mrnumber  = {1395371},
  doi       = {10.1007/978-1-4757-3845-2}
}

@book{Montgomery_Vaughan_2006,
  author    = {Montgomery, Hugh L. and Vaughan, Robert C.},
  title     = {Multiplicative Number Theory {I}: Classical Theory},
  series    = {Cambridge Studies in Advanced Mathematics},
  publisher = {Cambridge University Press},
  address   = {Cambridge},
  year      = {2006}
}

@book{Broughan2021,
  author    = {Broughan, Kevin},
  title     = {Bounded Gaps Between Primes: The Epic Breakthroughs of the Early Twenty-First Century},
  publisher = {Cambridge University Press},
  address   = {Cambridge},
  year      = {2021},
  isbn      = {978-1-108-83674-6},
  mrnumber  = {4412547},
  doi       = {10.1017/9781108872201}
}

@book{HalberstamRichert,
  author    = {Halberstam, H. and Richert, H.-E.},
  title     = {Sieve Methods},
  series    = {L.M.S. Monographs, No. 4},
  publisher = {Academic Press},
  address   = {London-New York},
  year      = {1974},
  isbn      = {0-12-318250-6}
}

@article{Romanov1934,
  author    = {Romanoff, N. P.},
  title     = {{\"U}ber einige {S}{\"a}tze der additiven {Z}ahlentheorie},
  journal   = {Mathematische Annalen},
  volume    = {109},
  number    = {1},
  pages     = {668--678},
  year      = {1934},
  publisher = {Springer},
  mrnumber  = {1512916},
  doi       = {10.1007/BF01449161}
}

@article{erdos1950integers,
  author    = {Erd{\H{o}}s, Paul},
  title     = {On integers of the form $2^k+p$ and some related problems},
  journal   = {Summa Brasiliensis Mathematicae},
  volume    = {2},
  number    = {8},
  pages     = {113--123},
  year      = {1950}
}

@article{crocker1971,
  author    = {Crocker, Roger},
  title     = {On the sum of a prime and of two powers of two},
  journal   = {Pacific Journal of Mathematics},
  volume    = {36},
  number    = {1},
  pages     = {103--107},
  year      = {1971},
  publisher = {Mathematical Sciences Publishers}
}

@article{sun2000integers,
  author    = {Sun, Zhi-Wei},
  title     = {On integers not of the form $c(2^a+2^b)+p$},
  journal   = {Acta Arithmetica},
  volume    = {93},
  number    = {3},
  pages     = {261--275},
  year      = {2000}
}

@article{ChenSun2004,
  author    = {Chen, Yong-Gao and Sun, Xue-Gong},
  title     = {On {Romanoff}'s constant},
  journal   = {Journal of Number Theory},
  volume    = {106},
  number    = {2},
  pages     = {275--284},
  year      = {2004},
  doi       = {10.1016/j.jnl.2003.11.009}
}

@article{pintz2006note,
  author    = {Pintz, J{\'a}nos},
  title     = {A note on {Romanov}'s constant},
  journal   = {Acta Mathematica Hungarica},
  volume    = {112},
  number    = {1-2},
  pages     = {1--14},
  year      = {2006},
  publisher = {Springer}
}

@article{Elsholtz2014,
  author    = {Elsholtz, Christian and Schlage-Puchta, Jan-Christoph},
  title     = {On {Romanov}'s constant},
  journal   = {Mathematische Zeitschrift},
  volume    = {288},
  number    = {3-4},
  pages     = {713--724},
  year      = {2018},
  doi       = {10.1007/s00209-017-1908-x}
}

@misc{sun2014anbnmodulom,
  author       = {Sun, Zhi-Wei},
  title        = {On $a^n+bn$ modulo $m$},
  howpublished = {arXiv preprint arXiv:1312.1166},
  year         = {2013}
}

@misc{OEIS,
  author    = {Sun, Zhi-Wei},
  title     = {Sequence {A231201}: Number of ways to write $n = x + y \ (x, y > 0)$ with $2^x + y$ prime},
  howpublished = {The On-Line Encyclopedia of Integer Sequences},
  year      = {2013},
  url       = {https://oeis.org/A231201}
}

@article{Wu2004,
  author    = {Wu, Jie},
  title     = {Chen's double sieve, {Goldbach}'s conjecture and the twin prime problem},
  journal   = {Acta Arithmetica},
  volume    = {114},
  number    = {3},
  pages     = {215--273},
  year      = {2004},
  publisher = {Instytut Matematyczny Polskiej Akademii Nauk},
  doi       = {10.4064/aa114-3-2}
}

@article{Habsieger2006,
  author    = {Habsieger, Laurent and Roblot, Xavier-Fran{\c{c}}ois},
  title     = {On integers of the form $p+2^k$},
  journal   = {Acta Arithmetica},
  volume    = {122},
  number    = {1},
  pages     = {45--50},
  year      = {2006},
  doi       = {10.4064/aa122-1-4}
}

@article{Gallagher1976,
  author    = {Gallagher, Patrick X.},
  title     = {On the distribution of primes in short intervals},
  journal   = {Mathematika},
  volume    = {23},
  number    = {1},
  pages     = {4--9},
  year      = {1976},
  publisher = {Cambridge University Press}
}
\end{document}